%% file: CoulembierRome.tex
\documentclass[graybox]{svmult}

\usepackage{amsmath}
\usepackage{amsxtra}
\usepackage{amscd}
\usepackage{amsfonts}
\usepackage{amssymb}

\usepackage{mathptmx}       
\usepackage{helvet}         
\usepackage{courier}        
\usepackage{type1cm}        
\usepackage{makeidx}         
\usepackage{graphicx}        
\usepackage{multicol}        
\usepackage[bottom]{footmisc}

\newcommand{\C}{{\mathbb C}}

\newcommand{\fa}{{\mathfrak a}}
\newcommand{\fb}{{\mathfrak b}}
\newcommand{\fg}{{\mathfrak g}}
\newcommand{\fh}{{\mathfrak h}}

\newcommand{\fn}{{\mathfrak n}}
\newcommand{\ofn}{\overline{\mathfrak n}}

\newcommand{\Hom}{{\rm Hom}}

\newcommand{\Ext}{{\rm Ext}}
\newcommand{\Span}{{\rm Span}}
\newcommand{\im}{{\rm im}}
\newcommand{\Res}{{\rm Res}}

\newcommand{\tto}{\twoheadrightarrow}

\newcommand{\U}{{\rm U}}

\newcommand{\oa}{\overline{0}}
\newcommand{\ob}{\overline{1}}

\newcommand{\cA}{{\mathcal A}}

\newcommand{\cP}{{\mathcal P}}

\newcommand{\cO}{{\mathcal O}}

\newcommand{\cR}{{\mathcal R}}

\newcommand{\ch}{\rm{ch}}

\makeindex             

\begin{document}

\title*{Homological algebra for osp(1/2n)}
\author{Kevin Coulembier}
\institute{Kevin Coulembier \at Department of Mathematical Analysis, Ghent University, Krijgslaan 281, 9000 Gent, Belgium, \email{coulembier@cage.ugent.be}}
%
%
\maketitle

\abstract{We discuss several topics of homological algebra for the Lie superalgebra $\mathfrak{osp}(1|2n)$. First we focus on Bott-Kostant cohomology, which yields classical results although the cohomology is not given by the kernel of the Kostant Laplace operator. Based on this cohomology we can derive strong Bernstein-Gelfand-Gelfand resolutions for finite dimensional $\mathfrak{osp}(1|2n)$-modules. Then we state the Bott-Borel-Weil theorem which follows immediately from the Bott-Kostant cohomology by using the Peter-Weyl theorem for $\mathfrak{osp}(1|2n)$. Finally  we calculate the projective dimension of irreducible and Verma modules in the category $\cO$.}

\section{Introduction}

The Lie superalgebra $\mathfrak{osp}(1|2n)$ plays an exceptional role in the theory of Lie superalgebras, see \cite{MR0486011}. Contrary to the other simple finite dimensional Lie superalgebras the Harish-Chandra map yields an isomorphism $Z(\fg)\cong S(\fh)^W$. Closely related to this observation is the fact that the category of finite dimensional representations is semisimple. In other words, all integral dominant highest weights are typical and every finite dimensional representation is completely reducible. As a consequence the algebra of regular functions on a Lie supergroup with superalgebra $\mathfrak{osp}(1|2n)$ satisfies a Peter-Weyl decomposition. 

Because of these extraordinary properties, the algebra $\mathfrak{osp}(1|2n)$ and its representation theory is relatively well-understood, see e.g. \cite{MR1784679, MR1479886, MR0648354}. In this paper we prove that certain standard topics of homological algebra for $\mathfrak{osp}(1|2n)$ allow elegant conclusions of the classical type. In particular the remarkable connection with the Lie algebra $\mathfrak{so}(2n+1)$, see e.g. \cite{MR0648354}, is confirmed.

First we focus on cohomology of the nilradical $\fn$ of the Borel subalgebra $\fb$ with values in finite dimensional $\mathfrak{osp}(1|2n)$-representations. Since the coboundary operator commutes with the Cartan subalgebra $\fh$ these cohomology groups are $\fh$-modules. For {\em Lie algebras} it can be proved that the cohomology is isomorphic to the kernel of the Kostant Laplace operator, see \cite{MR0142696}. This operator is equivalent to an element of $S(\fh)^W$. From the results in \cite{MR0578996, Bott, MR0142696} it then follows that every weight in the kernel of the Laplace operator (or equivalently in the cohomology) only appears with multiplicity one in the space of cochains. 

For {\em Lie superalgebras} in general the kernel of the Laplace operator is larger than the cohomology groups, see \cite{BGG}, even for $\mathfrak{osp}(1|2n)$ as we will see. We will also find that the weights appearing in the cohomology groups appear inside the space of cochains with multiplicities greater than one. We compute the cohomology by quotienting out an exact subcomplex, such that the resulting complex is isomorphic to that of $\mathfrak{so}(2n+1)$.

We use this result to obtain Bott-Borel-Weil (BBW) theory for $\mathfrak{osp}(1|2n)$. The classical BBW result in \cite{Bott} computes the sheaf cohomology on line bundles over the flag manifold of a semisimple Lie group. In general it is a difficult task to compute these cohomology groups for supergroups. BBW theory for the typical blocks was obtained in \cite{MR0957752}. Important further insight was gained in \cite{MR2734963, MR1036335, MR2059616}. 

Since all blocks are typical for $\mathfrak{osp}(1|2n)$ the BBW theorem for $\mathfrak{osp}(1|2n)$ is included in the results in \cite{MR0957752}. The $\fn$-cohomology results mentioned above could then be derived from the BBW result. Here we take the inverse approach because, despite being more computational, it clearly reveals the mechanism that makes the kernel of the Laplace operator larger than the cohomology groups, here caused by non-isotropic odd roots. When the kernel of the Laplace operator coincides with the cohomology it was proved in \cite{BGG} that the irreducible modules of basic classical Lie superalgebras have a strong Bernstein-Gelfand-Gelfand (BGG) resolution (see \cite{MR0578996}). 

In this paper we prove that finite dimensional modules of $\mathfrak{osp}(1|2n)$ always possess a strong BGG resolution. As can be expected from \cite{BGG} the main difficulty is dealing with the property that the kernel of the Kostant Laplace operator is larger than the cohomology. By making extensive use of the BGG theorem for $\mathfrak{osp}(1|2n)$ of \cite{MR1479886} and our result on $\fn$-cohomology we can overcome this difficulty. Other results on BGG resolutions for basic classical Lie superalgebras were obtained in \cite{MR2600694, Cheng, BGG, BBW}.

Finally we focus on the projective dimension of structural modules in the category $\cO$ for $\mathfrak{osp}(1|2n)$. The main result is that the projective dimension of irreducible and Verma modules with a regular highest weight is given in terms of the length of the element of Weyl group making them dominant. In particular we obtain that the global dimension of a regular block in $\cO$ is $2n^2$.

The remainder of the paper is organised as follows. In Section \ref{secPrel} we introduce some notations and conventions. The cohomology groups $H^k(\fn,-)$ are calculated in Section \ref{Kostant}. This result is then used in Section \ref{secBGG} to derive BGG resolutions. In Section \ref{secBBW} the $\fn$-cohomology result is translated into a BBW theorem. In Section \ref{secpd} the projective dimensions in the category $\cO$ are calculated. Finally there are two appendices. In Appendix 1 the technical details of the computation of the $\fn$-cohomology are given. In Appendix 2 we state some facts about the BGG category $\cO$ for basic classical Lie superalgebras.

\section{Preliminaries}
\label{secPrel}
For the complex basic classical Lie superalgebra $\fg=\mathfrak{osp}(1|2n)$ we consider the simple positive roots
\[\delta_1-\delta_2,\delta_2-\delta_3,\cdots,\delta_{n-1}-\delta_n,\delta_n\]
corresponding to the standard system of positive roots, see \cite{MR0486011}. For this system, the set of even positive roots is given by 
\[\Delta^+_{\oa}=\{\delta_i-\delta_j|1\le i<j\le n\}\cup\{\delta_i+\delta_j|1\le i\le j\le n\}\]
and the set of odd positive roots by
\[\Delta^+_{\ob}=\{\delta_i|1\le i\le n\}.\]
This leads to the value $\rho=\sum_{j=1}^n(n+\frac{1}{2}-j)\delta_j$ for half the difference between the sum of even roots and the sum of odd roots.

The Cartan subalgebra of $\mathfrak{osp}(1|2n)$ is denoted by $\fh$. The subalgebra consisting of positive (negative) root vectors is denoted by $\fn$ ($\ofn$). The corresponding triangular decomposition is given by $\mathfrak{osp}(1|2n)=\ofn+\fh+\fn$. The Borel subalgebra is denoted by $\fb=\fh+\fn$.

The Weyl group $W$ of $\mathfrak{osp}(1|2n)$ is the same as for the underlying Lie algebra $\mathfrak{sp}(2n)$ (and isomorphic to the Weyl group of $\mathfrak{so}(2n+1)$), where the action is naturally extended to include the odd roots of $\mathfrak{osp}(1|2n)$. By the dotted action of $w\in W$ on elements $\lambda\in \fh^\ast$ we mean the $\rho$-shifted action: $w\cdot\lambda=w(\lambda+\rho)-\rho$. Since the Weyl group is the same as for the underlying Lie algebra, the notion of the Chevallay-Bruhat ordering and the length $|w|=l(w)$ of an element $w\in W$, remains unchanged. However, the notion of strongly linked weights, see Section 5.1 in \cite{MR2428237} or Section 10.4 in \cite{MR2906817}, should be interpreted with respect to $\rho$ and not $\rho_{\oa}$ (half the sum of the positive roots of $\mathfrak{sp}(2n)$). Through the identification of the Weyl groups and root lattices of $\mathfrak{osp}(1|2n)$ and $\mathfrak{so}(2n+1)$, this shifted action coincides. In particular the characters of irreducible highest weights modules of $\mathfrak{osp}(1|2n)$ and $\mathfrak{so}(2n+1)$ coincide, see e.g. \cite{MR0648354}.

The set of integral dominant weights is denoted by $\cP^+\subset\fh^\ast$. For each $\lambda\in\fh^\ast$ the corresponding Verma module is denoted by $M(\lambda)=\U(\fg)\otimes_{\U(\fb)}\C_\lambda$. Where $\C_\lambda$ is the one dimensional $\fb$-module with properties $H\C_\lambda=\lambda(H)\C_\lambda$ for all $H\in\fh$ and $\fn\C_\lambda=0$. The quotient of $M(\lambda)$ with respect to its unique maximal submodule is irreducible and denoted by $L(\lambda)$. The module $L(\lambda)$ is finite dimensional if and only if $\lambda\in\cP^+$. For each $\mu\in\fh^\ast$ we denote the central character associated with it by $\chi_\mu: Z(\fg)\to\C$.

The spaces of $k$-chains for $\ofn$-homology in an $\mathfrak{osp}(1|2n)$-module $V$ are denoted by $C_k(\ofn,V)=\Lambda^k\ofn\otimes V$. These spaces are naturally $\fh+\ofn$-modules where the action is the tensor product of the adjoint action and the restricted action on the $\mathfrak{osp}(1|2n)$-module $V$. The boundary operator $\delta^\ast_k:C_k(\ofn,V)\to C_{k-1}(\ofn,V)$ is defined by
\begin{eqnarray*}
\delta^\ast_k(Y\wedge f)&=& -Y\cdot f-Y\wedge\delta^\ast_{k-1}(f)\quad\mbox{and}\quad \delta^\ast_0=0,
\end{eqnarray*}
for $Y\in\ofn$ and $f\in C_{k-1}(\ofn,V))$, see e.g. \cite{BGG}. This operator is an $\fh$-module morphism and satisfies $\delta_{k}^\ast\circ\delta_{k+1}^\ast=0$. The homology groups are defined as $H_k(\ofn,V)=\ker\delta_k^\ast/\im\delta_{k+1}^\ast$ and are naturally $\fh$-modules. A general approach to the concept of Lie superalgebra cohomology can e.g. be found in Chapter 16 in \cite{MR1479886}.

For an abelian category $\cA$, the right derived functors (see \cite{MR1269324}) of the left exact functor given by $\Hom_{\cA}(A,-)$, for $A$ an object of $\cA$, are denoted by $\Ext^k_{\cA}(A,-)$, where  $\Ext^1_{\cA}(A,-)$ is also written as $\Ext_{\cA}(A,-)$. When the category of finitely generated $\fa$-modules is considered, for some algebra $\fa$, the name of the category is replaced by $\fa$.

\section{Bott-Kostant cohomology}
\label{Kostant}

The main result of this section is the following description of the homology and cohomology of the nilradical of the Borel subalgebra of $\mathfrak{osp}(1|2n)$ or its dual, with values in irreducible representations of $\mathfrak{osp}(1|2n)$.
\begin{theorem}
\label{Kostantosp12n}
The (co)homology of $\fn$ and $\ofn$ in the irreducible finite dimensional $\mathfrak{osp}(1|2n)$-representation $L(\lambda)$ is given by
\begin{eqnarray*}
H^k(\fn,L(\lambda))=\bigoplus_{w\in W(k)} \C_{w\cdot\lambda}&& \quad H_k(\fn,L(\lambda))=\bigoplus_{w\in W(k)} \C_{-w\cdot\lambda}\\
H^k(\ofn,L(\lambda))=\bigoplus_{w\in W(k)} \C_{-w\cdot\lambda}&&\quad H_k(\ofn,L(\lambda))=\bigoplus_{w\in W(k)} \C_{w\cdot\lambda},
\end{eqnarray*}
with $W(k)$ the set of elements of the Weyl group with length $k$, see Section 0.3 in \cite{MR2428237}.
\end{theorem}
One of these results implies the other three according to Lemma 6.22 in \cite{MR3012224}, Lemma 4.6 in \cite{BGG} or Theorem 17.6.1 in \cite{MR2906817}. The remainder of this section is devoted to proving the property for the $\ofn$-homology, where the more technical steps in the proof are given in Appendix 1.

For each root the corresponding space of root vectors is one dimensional. For each positive root $\alpha\in\Delta^+$, we fix one root vector with weight $-\alpha$ and denote it by $Y_\alpha\in\ofn$. We choose the normalisation such that $[Y_{\delta_i},Y_{\delta_i}]=Y_{2\delta_i}$ holds. Each element  $f\in C_d(\ofn,V)$ of the form $f=Y_{\alpha_1}\wedge\cdots \wedge Y_{\alpha_d}\otimes v$ for certain positive roots $\alpha_1,\cdots,\alpha_d$ and $v\in V$ is called a monomial. For convenience $v$ will often be considered to be a weight vector. We say that $f=Y_{\alpha_1}\wedge\cdots \wedge Y_{\alpha_d}\otimes v$ contains a monomial $Y_{\beta_1}\wedge\cdots \wedge Y_{\beta_k}\in\Lambda^k\ofn$ if $\{\beta_1,\cdots,\beta_k\}\subset\{\alpha_1,\cdots,\alpha_d\}$.

\begin{definition}
\label{DefAB}
The $\fh$-submodule of $C_\bullet(\ofn,V)$ spanned by all monomials that do not contain any $Y_{2\delta_i}$ or $Y^{\wedge 2}_{\delta_i}$ for $i\in\{1,\cdots,n\}$ is denoted by $R_\bullet(\ofn,V)$ and the subvectorspace spanned by all monomials that do contain a $Y_{2\delta_i}$ or $Y^{\wedge 2}_{\delta_i}$ is denoted by $W_\bullet(\ofn,V)$, then 
\[C_\bullet(\ofn,V)=R_\bullet(\ofn,V)\oplus W_\bullet(\ofn,V).\] 
The subspaces $A_\bullet^{(j)}$ and $B_\bullet^{(j)}$ of $W_\bullet(\ofn,V)$ are defined as
\begin{eqnarray*}
A_\bullet^{(j)}&=&\Span\{Y^{\wedge 2}_{\delta_j}\wedge f| \mbox{ f }\in C_\bullet(\ofn,V) \mbox{ contains no } Y_{2\delta_j}, \, Y_{2\delta_i} \mbox{ or } Y^{\wedge 2}_{\delta_i} \mbox{ for }i<j\}\\
B_\bullet^{(j)}&=&\Span\{Y_{2\delta_j}\wedge f| \mbox{ f }\in C_\bullet(\ofn,V) \mbox{ contains no } Y_{2\delta_i} \mbox{ or } Y^{\wedge 2}_{\delta_i} \mbox{ for }i<j\}.
\end{eqnarray*}
\end{definition}

The subspace $R_\bullet(\ofn,L(\lambda))$ for $\lambda\in\cP^+\subset\fh^\ast$ is isomorphic as an $\fh$-module to the the corresponding full spaces of chains for the nilradical of $\mathfrak{so}(2n+1)$ and the corresponding representation of $\mathfrak{so}(2n+1)$ with the same highest weight $\lambda$. In particular, $R_k(\ofn,L(\lambda))=0$ for $k>n^2$. 

Using the results in Appendix 1 we can prove that the homology of $C_\bullet(\ofn,V)$ can essentially be described in terms of $R_\bullet(\ofn,V)$. This is based on the fact that the homology of a complex does not change after quotienting out an exact complex:

\begin{proposition}
\label{subcomplex}
Let $S_\bullet\subset C_\bullet(\ofn,V)$ be an exact subcomplex (and $\fh$-submodule). The operator $d:C_\bullet(\ofn,V)/S_\bullet\to C_\bullet(\ofn,V)/S_\bullet$ canonically induced from $\delta^\ast$ satisfies 
\begin{eqnarray*}
H_\bullet(C/S)\cong\ker d_k /\im d_{k+1} \cong \ker \delta^\ast_k/\im \delta^\ast_{k+1}\cong H_\bullet(C)
\end{eqnarray*}
as $\fh$-modules.
\end{proposition}
\begin{proof}
The operator $d$ is defined as $d(f+S)=\delta^\ast (f)+S$ for $f\in C_\bullet(\ofn,V)$. The morphism 
\begin{eqnarray*}
\eta:\ker \delta^\ast\to \ker d&& \eta(f)=f+S
\end{eqnarray*}
is well-defined. Since $\eta(\im \delta^\ast)\subset \im d$ this descends to a morphism $\widetilde{\eta}:\ker \delta^\ast/\im \delta^\ast\to \ker d/\im d$. 

We prove that this is injective. Assume that $f\in \ker \delta^\ast \backslash \im\delta^\ast$, we have to prove that $f$ is not of the form $\delta^\ast(g)+s$ for $s\in S_\bullet$. If $f$ were of this form it immediately would follow that $s\in\ker\delta^\ast\cap S_\bullet=\im \delta^\ast\cap S_\bullet$ and therefore $f\in\im\delta^\ast$, which is a contradiction.

Finally we prove that $\widetilde{\eta}$ is also surjective. Every element in $\ker d_k /\im d_{k+1}$ is represented by some $a\in C_\bullet(\ofn,V)$ such that $\delta^\ast a=s\in S_\bullet$. Since $\delta^\ast s=0$ and $S_\bullet$ forms an exact complex, there is a certain $s_1\in S_\bullet$ such that $s=\delta^\ast s_1$. The element $a-s_1$ is clearly inside $\ker\delta^\ast$, so the proposition follows from $\eta(a-s_1)=a+S_\bullet$.
\end{proof}

\begin{theorem}
\label{CtoR}
For any $\mathfrak{osp}(1|2n)$-module $V$, the subspace $A_\bullet\subset C_\bullet(\ofn,V)$ satisfies $A_\bullet\cap \ker\delta^\ast=\{0\}$,
\begin{eqnarray*}
\left(A_\bullet\oplus\delta^\ast A_\bullet \right)\cap R_\bullet(\ofn,V)=\{0\}&\mbox{and}&A_\bullet\oplus\delta^\ast A_\bullet \oplus R_\bullet(\ofn,V)=C_\bullet(\ofn,V).
\end{eqnarray*}
\end{theorem}
\begin{proof}
The proof makes use of the results in Lemma \ref{spacesAB} and Theorem \ref{thmisom} in Appendix 1.

The property $A_\bullet\cap \ker\delta^\ast=\{0\}$ follows immediately from Theorem \ref{thmisom}. This implies that $A_\bullet\oplus\delta^\ast A_\bullet $ is in fact a direct sum, since $\delta^\ast A_\bullet \subset \ker\delta^\ast$. 

If $r=f+g$ with $r\in R_\bullet(\ofn,V)$, $f\in A_\bullet$ and $g\in \delta^\ast A_\bullet$, then $\phi(h)=0$ for any $h\in A_\bullet$ such that $\delta^\ast h=g$, with $\phi$ the isomorphism defined in Theorem \ref{thmisom}, and therefore $g=0$. Since $A_\bullet\subset W_\bullet(\ofn,V)$, $r=f$ implies $r=0=f$ according to Definition \ref{DefAB} and we obtain $\left(A_\bullet\oplus\delta^\ast A_\bullet \right)\cap R_\bullet(\ofn,V)=\{0\}$.

The last property follows from the previous one and dimensional considerations. The first property in the theorem \ref{thmisom} implies that $\dim A_{k+1} =\dim \delta^\ast A_{k+1}=\dim \left(\delta^\ast A\right)_{k}$, together with Lemma \ref{spacesAB} this yields $\dim \left(\delta^\ast A\right)_k=\dim B_{k}$. Therefore $\dim A_k+\dim\left( \delta^{\ast}A\right)_{k}+\dim R_k(\fn,V)=\dim C_k(\ofn,V)$ according to Definition \ref{DefAB} and Lemma \ref{spacesAB}.
\end{proof}

\begin{remark}
Thus far the fact that $V$ is not just an $\fh+\ofn$-module but also an $\mathfrak{osp}(1|2n)$-module has not been used. Theorem \ref{CtoR} could therefore be used to calculate $\ofn$-homology with values in arbitrary finite dimensional $\fh+\ofn$-modules.
\end{remark}

Now we can give the proof of Theorem \ref{Kostantosp12n}.
\begin{proof}

We calculate the Euler characteristic of the homology: $\sum_{i=0}^\infty (-1)^i \ch H_i(\ofn,L(\lambda))$
\begin{eqnarray*}
&=&\sum_{i=0}^\infty (-1)^i \ch(\Lambda^i\ofn)\ch L(\lambda)\\
&=&\frac{\prod_{\alpha\in\Delta^+_{\oa}}(1-e^{-\alpha})}{\prod_{\gamma\in\Delta^+_{\ob}}(1+e^\gamma)}\frac{\prod_{\gamma\in\Delta^+_{\ob}}(e^{\gamma/2}+e^{-\gamma/2})}{\prod_{\alpha\in\Delta^+_{\oa}}(e^{\alpha/2}-e^{-\alpha/2})}\sum_{w\in W}(-1)^{|w|}e^{w(\lambda+\rho)}\\
&=&\sum_{w\in W}(-1)^{|w|}e^{w\cdot\lambda},
\end{eqnarray*}
which is the technique through which Kostant obtained the Weyl character formula from this type of cohomology in \cite{MR0142696}.

Now from Section 4 in \cite{BGG} it follows that $H_k(\ofn,L(\Lambda))\subset\ker\Box$ with $\Box$ the Kostant Laplacian $\Box$ on $C_k(\ofn,V)$. This operator $\Box$ is a quadratic element of $\U(\fh)$. From Proposition \ref{subcomplex} and Theorem \ref{CtoR} it follows that this property can be made stronger to $H_k(\ofn,L(\Lambda))\subset\ker\Box_{R_k}$. The $\fh$-module $R_k$ is isomorphic to the chains for Bott-Kostant cohomology for $\mathfrak{so}(2n+1)$. There the cohomology is well-known and equal to the kernel of the Kostant Laplace operator, which takes the same form as for $\mathfrak{osp}(1|2n)$. Therefore the result in \cite{MR0142696} and the observation of the connection between $\mathfrak{osp}(1|2n)$ and $\mathfrak{so}(2n+1)$ in Section \ref{secPrel} yields
\[H_k(\ofn, L(\lambda))\subset\bigoplus_{w\in W(k)}\C_{w\cdot\lambda}.\]
The Euler characteristic then implies that these inclusions must be equalities.
\end{proof}

The results on cohomology of $\fn$ can be reinterpreted in terms of $\Ext$-functors in the category $\cO$ as defined in Appendix 2.
 \begin{corollary}
For $\fg=\mathfrak{osp}(1|2n)$, $\lambda\in \cP^+$ and $\mu\in\fh^\ast$, the property
\[\Ext^k_{\cO}(M(\mu),L(\lambda))=\begin{cases}1&\mbox{if }\mu=w\cdot\lambda\mbox{ with }|w|=k\\
0&\mbox{otherwise}\end{cases}\]
holds.
\end{corollary}
\begin{proof}
As in the classical case the Frobenius reciprocity $\Hom_\cO(\U(\fg)\otimes_{\U(\fb)}\C_\mu,V)=\Hom_{\fb}(\C_\mu,\Res^{\fg}_{\fb}V)$ holds for all $V\in\cO$. This gives an equality of functors $\cO\to\rm{Sets}$ and since the functor $\Res^{\fg}_{\fb}$ is exact we can take the right derived functor of both left exact functors above to obtain
\[\Ext^k_\cO(M(\mu),V)=\Ext^k_{\fb}(\C_\mu,\Res^{\fg}_{\fb}V)\]
If we use $\Hom_{\fb}(\C_\mu,-)=\Hom_{\fh}(\C_\mu,-)\circ \Hom_{\fn}(\C,-)$, the fact that $\Hom_{\fh}(\C_\mu,-)$ is exact and $\Ext^k(\fn,-)=H^k(\fn,-)$, see Lemma 4.7 in \cite{BGG}, we obtain
\[\Ext^k_\cO(M(\mu),V)=\Hom_{\fh}\left(\C_\mu, H^k(\fn,V)\right).\]
The corollary then follows from Theorem \ref{Kostantosp12n}.
\end{proof}

\section{Bernstein-Gelfand-Gelfand resolutions}
\label{secBGG}

The main result of this section is that all finite dimensional modules of $\mathfrak{osp}(1|2n)$ can be resolved in terms of direct sums of Verma modules. Such resolutions are known as (strong) BGG resolutions and were discovered first for semisimple Lie algebras in \cite{MR0578996}.
\begin{theorem}
\label{resolosp12n}
Every finite dimensional representation $L(\lambda)$ of $\mathfrak{osp}(1|2n)$ has a resolution in terms of Verma modules of the form
\begin{eqnarray*}
&&0\to \bigoplus _{w\in W(n^2)} M(w\cdot\lambda)\to\cdots\to \bigoplus _{w\in W(j)} M(w\cdot\lambda)\to\cdots\\
&&\to \bigoplus _{w\in W(1)} M(w\cdot\lambda)\to M(\lambda)\to L(\lambda)\to 0.
\end{eqnarray*}
\end{theorem}
In the remainder of this section we provide the results needed to prove Theorem \ref{resolosp12n}. We will make extensive use of the notions and results on the category $\cO$ in Appendix 2.

First, we state the BGG theorem for $\mathfrak{osp}(1|2n)$, which was proved by Musson in Theorem 2.7 in \cite{MR1479886}:
\begin{theorem}[BGG theorem]
\label{thmMusson}
For $\fg=\mathfrak{osp}(1|2n)$ and $\lambda,\mu\in\fh^\ast$ it holds that $[M(\lambda):L(\mu)]\not=0$ if and only if $\mu \uparrow \lambda$ ($\mu$ is strongly linked to $\lambda$).
\end{theorem}
Using this we obtain the following corollary.
\begin{corollary}
\label{extVerma}
Consider $\fg=\mathfrak{osp}(1|2n)$ and $\mu,\lambda\in\fh^\ast$. If $\Ext_{\cO}\left(M(\mu),M(\lambda)\right)\not=0$ then $\mu\uparrow\lambda$ but $\mu\not=\lambda$.
\end{corollary}
\begin{proof}
The property $\Ext_{\cO}\left(M(\mu),M(\lambda)\right)\not=0$ holds if and only if there is a short exact non-split sequence of the form $M(\lambda)\hookrightarrow M\tto M(\mu)$ for an $M\in\cO$. That $\mu\not=\lambda$ must hold follows immediately from the fact that otherwise $M$ would contain two highest weight vectors of weight $\lambda$, which both generate a Verma module.

The remainder of the proof is then equivalent with the proof of Theorem 6.5 in \cite{MR2428237}. We consider the projective cover $P(\mu)$ of $M(\mu)$, which exists and has a standard filtration by Lemma \ref{projO}. This filtration $0= P_0\cdots\subset P_1\subset\cdots P_n=P$ satisfies $P_i/P_{i-1}\cong M(\mu_i)$ with $\mu\uparrow\mu_i$ by the combination of Theorem \ref{thmMusson} and Lemma \ref{BGGreproc}.

The canonical map $P(\mu)\to M(\mu)$ extends to $\phi:P(\mu)\to M$ and since the exact sequence does not split we obtain that for some $i$, $\phi(P_i)\cap M(\lambda)\not=0$ while $\phi(P_{i-1})\cap M(\lambda)=0$. This implies that $M(\lambda)$ has a nonzero submodule which is a homomorphic image of $M(\mu_i)$ and therefore $[M(\lambda):L(\mu_i)]\not=0$. Applying Theorem \ref{thmMusson} again yields $\mu_i\uparrow \lambda$.

These two results lead to $\mu\uparrow\lambda$.
\end{proof}

Now we can prove the following consequence of this corollary.
\begin{lemma}
\label{extzero}
Consider $w\in W$, $\lambda\in\cP^+$ and a module $M$ with a standard filtration where the occurring Verma modules are of the form $M(w'\cdot\lambda)$ with $l(w')\ge l(w)$, then
\[\Ext_\cO(M(w\cdot\lambda),M)=0.\]
Furthermore, any module $S$ in $\cO_{\chi_\lambda}$ with standard filtration has a filtration of the form $S=S^{(0)}\supseteq S^{(1)}\supseteq\cdots S^{(n^2)}\supseteq S^{(n^2+1)}= 0$, where $S^{(j)}/S^{(j+1)}$ is isomorphic to the direct sum of Verma modules with highest weights $u\cdot\lambda$ with $u\in W(n^2-j)$.
\end{lemma}
\begin{proof}
The first statement is an immediate application of Corollary \ref{extVerma} if $M$ is a Verma module. The remainder can then be proved by induction on the filtration length. Assume it is true for filtration length $p-1$ and $M$ has filtration length $p$. Then there is a short exact sequence
\[0\to N\to M\to M(w_p\cdot \lambda)\to0\]
for $N$ having a standard filtration of the prescribed kind of length $p-1$ and $l(w_p)\ge l(w)$. Applying the functor $\Hom_\cO(M(w\cdot\lambda),-)$ and its right derived functors gives a long exact sequence
\begin{eqnarray*}
&&0\to\Hom_\cO(M(w\cdot\lambda),N)\to \Hom_\cO(M(w\cdot\lambda),M)\to \Hom_\cO(M(w\cdot\lambda),M(w_p\cdot \lambda))\\
&&\to\Ext_\cO(M(w\cdot\lambda),N)\to \Ext_\cO(M(w\cdot\lambda),M)\to \Ext_\cO(M(w\cdot\lambda),M(w_p\cdot \lambda))\to\cdots.
\end{eqnarray*}
Since $\Ext_\cO(M(w\cdot\lambda),N)=0= \Ext_\cO(M(w\cdot\lambda),M(w_p\cdot \lambda))$ by the induction step we obtain $\Ext_\cO(M(w\cdot\lambda),M)=0$.

In order to prove the second claim we consider an arbitrary module $K$ in $\cO_{\chi_\lambda}$ with a standard filtration, 
\[K=K_0\supset K_1\supset \cdots \supset K_d=0\qquad \mbox{with } \quad K_i/K_{i+1}\cong M(w_{(i)}\cdot\lambda).\]
Consider an arbitrary $i$ such that $w_{(i)}$ has the minimal length appearing in the set $\{w_{(j)},j=0,\cdots,d-1\}$, since $\Ext_\cO(M(w_{(i)}\cdot\lambda), K_{i+1})=0$ by the first part of the lemma it follows that $M(w_{(i)})\subset K_i\subset K$. Therefore the direct sum of all these Verma modules are isomorphic to a submodule of $K$. This submodule can be quotiented out and the statement follows by iteration.
\end{proof}

As in \cite{MR0578996} we start by constructing a resolution of $L(\lambda)$ in terms of modules induced by the spaces of chains 
\[C_\bullet(\ofn,L(\lambda))\cong \Lambda^\bullet\ofn\otimes L(\lambda)\cong \Lambda^\bullet\left(\fg/\fb\right)\otimes L(\lambda),\] which will possess standard filtrations by construction. For the classical case, restricting to the block in $\cO$ which $L(\lambda)$ belongs to, exactly reduces from $C_k(\ofn,V)$ to $H_k(\ofn,V)$. Corollary \ref{extVerma} then already yields the BGG resolutions. In fact, according to the results in \cite{MR0142696} only one Casimir operator is needed for this reduction, the quadratic one. Applying this procedure in the case of Lie superalgebras would however lead to a resolution in terms of the kernel of the Laplace operator, which is still larger than the homology groups, as discussed in Section \ref{Kostant}. In case the kernel of the Laplace operator agrees with the cohomology, strong BGG resolutions for basic classical Lie superalgebras always exist, according to the result in \cite{BGG}.

\begin{lemma}
\label{preBGG}
For each finite dimensional representation $L(\lambda)$ of $\mathfrak{osp}(1|2n)$, there is a finite resolution of the form
\begin{eqnarray*}
\cdots\to D_k\to\cdots\to D_1\to D_0\to L(\lambda)\to 0,
\end{eqnarray*}
where each $D_k$ has a standard filtration. Moreover $D_k$ has a filtration $D_k=S_k^{(0)}\supseteq S_k^{(1)}\supseteq\cdots S_k^{(n^2)}\supseteq S_k^{(n^2+1)}= 0$, where $S_k^{(j)}/S_k^{(j+1)}$ is isomorphic to the direct sum of Verma modules with highest weights $w\cdot\lambda$ with $l(w)=n^2-j$.
\end{lemma}
\begin{proof}
The first step of the construction is parallel to the classical case. We can define an exact complex of $\fg$-modules of the form
\begin{eqnarray*}
&&\cdots\to \U(\fg)\otimes_{\U(\fb)}(\Lambda^k\fg/\fb\otimes L(\lambda))\to\cdots\to\\
&&  \U(\fg)\otimes_{\U(\fb)}(\Lambda^1\fg/\fb\otimes L(\lambda))\to  \U(\fg)\otimes_{\U(\fb)}L(\lambda)\to L(\lambda)\to 0
\end{eqnarray*}
where the maps are given by the direct analogs of those in \cite{MR0578996}, or Section 6.3 in \cite{MR2428237}, see also equation (4.1) in \cite{Cheng}. In fact it suffices to do this for $L(\lambda)$ trivial, since a straightforward tensor product can be taken afterwards.

Now we can restrict the resolution to the block of the category $\cO$ corresponding to the central character $\chi_\lambda$, which still yields an exact complex. Lemma \ref{extzero} then implies that the modules which appear must be of the proposed form.

It remains to be proved that the resolution is finite. This follows from the observation that for $k$ large enough all the weights appearing in $C_k(\ofn,L(\Lambda))$ are lower than those in the set $\{w(\lambda+\rho)-\rho|w\in W\}$.
 \end{proof}
 
Now we can prove Theorem \ref{resolosp12n}. Contrary to the classical case in \cite{MR0578996}, where the BGG resolutions are constructed to obtain an alternative derivation for the Bott-Kostant cohomology groups, we will need our result on the $\ofn$-homology to derive the BGG resolutions.
\begin{proof}
Since the modules appearing in the resolution in Lemma \ref{preBGG} have a filtration in terms of Verma modules, this corresponds to a projective resolution in the category of $\ofn$-modules. This can therefore be applied to calculate the right derived functors of the left exact contravariant functor $\Hom_{\ofn}(-,\C)$ acting on $L(\lambda)$, see \cite{MR1269324}. These functors satisfy $\Ext^k_{\ofn}(L(\lambda),\C)=H_k(\ofn,L(\lambda))^\ast$, see Lemma 4.6 and Lemma 4.7 in \cite{BGG}. 
By applying this we obtain that the homology $H_k(\ofn,L(\lambda))$ is equal to the homology of the finite complex of $\fh$-modules
\[\cdots\to D_k/(\ofn \,D_k)\to\cdots\to D_1/(\ofn\,D_1)\to D_0/(\ofn\, D_0)\to  0,\]
where the maps are naturally induced from the ones in Lemma \ref{preBGG}.

The $\fh$-modules $D_k/(\ofn \,D_k)$ are exactly given by all the highest weights of the Verma modules appearing in the standard filtration of $D_k$. We take the largest $k$ such that the filtration of $D_k$ contains a Verma module of highest weight $w\cdot\lambda$ with $l(w)<k$. Since such a weight can not be in $H_k(\ofn,L(\lambda))$ by Theorem \ref{Kostantosp12n}, it is not in the kernel of the mapping $ D_k/(\ofn \,D_k)\to  D_{k-1}/(\ofn \,D_{k-1})$ (it is not in the image of $ D_{k+1}/(\ofn \,D_{k+1})\to  D_{k}/(\ofn \,D_{k})$ since we chose $k$ maximal). We fix such a $w\cdot\lambda$ for $D_k$ with minimal $l(w)$. According to Lemma \ref{preBGG} $M(w\cdot\lambda)$ is actually a submodule of $D_k$. Under the $\fg$-module morphism in Lemma \ref{preBGG} this submodule is mapped to a submodule in $D_{k-1}$. The highest weight vector of $M(w\cdot \lambda)$ is mapped to a highest weight vector in $D_{k-1}$. Since the projection onto $D_{k-1}/(\ofn D_{k-1})$ is not zero this highest weight vector is not inside another Verma module. This implies that the quotient of $D_{k-1}$ with respect to the image of $M(w\cdot \lambda)$ still has a standard filtration. The appearance of $M(w\cdot \lambda)$ in $D_k$ and $D_{k-1}$ forms an exact subcomplex which can be quotiented out and according to Proposition \ref{subcomplex} the resulting complex is still exact. 

This procedure can be iterated until the resolution in Lemma \ref{preBGG} is reduced to a resolution of the form of Lemma \ref{preBGG} for which we use the same notations and where it holds that $S_k^{(j)}=0$ if $j>n^2-k$. In a similar step we can quotient out the Verma submodules of $S_k^{(n^2-k)}$ that do not contribute to $H_k(\ofn,L(\lambda))$. 

Then we can focus on the submodules $S_k^{(n^2-k)}\subset D_k$ of the resulting complex. Because of the link with the $\ofn$-homology each of the highest weight vectors of the Verma modules is not mapped to the highest weight vector of a Verma module in the filtration of $D_{k-1}$. Theorem \ref{thmMusson} implies that the image of a Verma module in $S_k^{(n^2-k)}$ under the composition of the map in Lemma \ref{preBGG} with the projection onto $D_{k-1}/S_{k-1}^{(n^2-k+1)}$ must be zero since the filtration of $D_{k-1}/S_{k-1}^{(n^2-k+1)}$ contains only Verma modules with highest weight $u\cdot\lambda$ with $l(u)\ge k$. So $S_k^{(n^2-k)}$ gets mapped to $S_{k-1}^{(n^2-k+1)}\subset D_{k-1}$, and thus there is a subcomplex of the desired form in Theorem \ref{resolosp12n}. The complex originating from quotienting out this subcomplex is exact, which can again be seen from the connection with $\ofn$-homology. Therefore we obtain that the subcomplex of the modules $S_k^{(n^2-k)}$ must also be exact and Theorem \ref{resolosp12n} is proven.
\end{proof}

\section{Bott-Borel-Weil theory}
\label{secBBW}

In this section we use the algebraic reformulation of the result of Bott, Borel and Weil in \cite{Bott} for algebraic groups to describe the Bott-Borel-Weil theorem for the algebraic supergroup $OSp(1|2n)$. This rederives the result for $\mathfrak{osp}(1|2n)$ in Theorem 1 in \cite{MR0957752}.

\begin{theorem}
\label{BBWosp12n}
Consider $\fg=\mathfrak{osp}(1|2n)$ and $\C_{\lambda}$ the irreducible $\fb$-module with $\fh \C_{-\lambda}=-\lambda(\fh)\C_{-\lambda}$.
\begin{itemize}
\item If $\lambda$ is regular, there exists a unique element of the Weyl group $W$ rendering $\Lambda:=w(\lambda+\rho)-\rho$ dominant. In this case
\[H^k(G/B,G\times_B\C_{-\lambda})=\begin{cases}L({\Lambda})&\mbox{if} \quad |w|=k\\ 0 &\mbox{if}\quad  |w|\not=k  \end{cases}.\]
\item If $\lambda$ is not regular, $H^k(G/B,G\times_B\C_{-\lambda})=0$.
\end{itemize}
\end{theorem}
\begin{proof}
For any $\fb$-module the holomorphic sections of the flag manifold satisfy $H^0(G/B,G\times_BV)=\Hom_{\fb}(\C,V\otimes \cR)$ with the $\fg\times\fg$-module $\cR$ given by the algebra of regular functions (the finite dual of the Hopf algebra $\U(\fg)$) on $OSp(1|2n)$, see the proof of Lemma 2 in \cite{MR2734963}. This algebra corresponds to the finite dual of the super Hopf algebra $\U(\fg)$. The derived functors therefore satisfy $H^k(G/B,G\times_BV)=\Ext^k_{\fb}(\C,V\otimes \cR)$. Since $\cR$ corresponds to the algebra of matrix elements, see Lemma 3.1 in \cite{MR2059616} and the category of finite dimensional $\mathfrak{osp}(1|2n)$-representations is semisimple the Peter-Weyl type theorem
\[\cR=\bigoplus_{\Lambda\in\cP^+}L(\Lambda)\times L(\Lambda)\]
follows immediately. Therefore BBW theory is expressed as
\[H^k(G/B,G\times_B \C_{-\lambda})=\bigoplus_{\Lambda\in\cP^+}\Hom_{\fh}(\C_{\lambda},H^k(\fn,L(\Lambda)))L(\Lambda).\]
The result then follows from Theorem \ref{Kostantosp12n}.
\end{proof}

\section{Projective dimension in $\cO$ of simple and Verma modules}
\label{secpd}

In this section we calculate projective dimensions of simple and Verma modules in $\cO$, which also gives the global dimension of the category $\cO$. For semisimple Lie algebras this was obtained by Mazorchuk in a general framework to calculate projective dimensions of structural modules in \cite{MR2366357}. Part of this approach extends immediately to $\mathfrak{osp}(1|2n)$, where the global dimension can actually be calculated from the BGG resolutions in Theorem \ref{resolosp12n}. However here, we follow an approach similar to the classical one sketched in Section 6.9 in \cite{MR2428237}.

\begin{theorem}
\label{projdim}
For $\fg=\mathfrak{osp}(1|2n)$ and $\lambda\in\cP^+$, the following equalities on the projective dimensions hold:
\begin{eqnarray*}
(i)&& p.d. M(w\cdot \lambda)= l(w)\\
(ii)&& p.d. L(w\cdot \lambda)= 2n^2-l(w)\\
(iii)&& gl.d. \cO_{\chi_\lambda}=2n^2.
\end{eqnarray*}
\end{theorem}
\begin{proof}
By Lemma \ref{Vermaproj}, statement $(i)$ is true for $w=1$, or $l(w)=0$. Then we proceed by induction on the length of $w$.

We use the general fact that if there is a short exact sequence of the form $A\hookrightarrow B \tto C$ then 
\[p.d. A\le \max\{p.d. B, p.d.C-1\}\qquad\mbox{and}\qquad p.d. C\le \max\{p.d. A+1, p.d.B\},\]
see \cite{MR2428237,MR1269324}.

Assume $(i)$ holds for all $w$ such that $l(w)<k$, then we take some $w\in W(k)$ and denote the kernel of the canonical morphism $P(w\cdot\lambda)\tto M(w\cdot\lambda)$ by $N$. The module $N$ has a standard filtration and the components can be obtained from the combination of Theorem \ref{thmMusson} and Lemma \ref{BGGreproc}. Therefore we obtain $p.d. N=l(w)-1$. The short exact sequence $N\hookrightarrow P(w\cdot\lambda)\tto M(w\cdot\lambda)$ implies $p.d. N \le p.d. M(w\cdot\lambda)-1$ and $p.d. M(w\cdot\lambda)\le p.d. N+1$ and we obtain $p.d.M(w\cdot\lambda)=k$.

This proves $(i)$. The result of $(i)$ implies $(ii)$ for $l(w)=n^2$ since then $M(w\cdot\lambda)=L(w\cdot\lambda)$ by Theorem \ref{thmMusson}. From this point on statement $(ii)$ can also be proved by induction, now using the short exact sequences of the form $N'\hookrightarrow M(w\cdot\lambda)\tto L(w\cdot\lambda)$ with $N'$ the unique maximal submodule of $M(w\cdot\lambda)$.

The result of $(ii)$ immediately implies $(iii)$.
\end{proof}

\begin{remark}
Since projective modules in the category $\cO$ have a standard filtration, see Lemma \ref{projO}, a projective resolution of $V$ provides a complex with homology $H_k(\ofn,V)$, for any basic classical Lie superalgebra. In particular it follows that the projective dimension of $V$ in the category $\cO$ is larger than or equal to the projective dimension as an $\ofn$-module. In fact, for $\mathfrak{osp}(1|2n)$, using the technique from the proof of Proposition 2 in \cite{MR2366357}, and the result in Theorem \ref{resolosp12n}, one obtains that the projective dimension in $\cO$ is at least twice the projective dimension as an $\ofn$-module. The result for $\fg=\mathfrak{osp}(1|2n)$ in Theorem \ref{projdim} exactly states that this bound is actually an equality.
\end{remark}

\begin{acknowledgement}
The author is a Postdoctoral Fellow of the Research Foundation - Flanders (FWO). The author wishes to thank Ruibin Zhang for fruitful discussions on this topic. \end{acknowledgement}

\section*{Appendix 1: Structure of the space of chains $C_\bullet(\ofn,V)$}
%
%

In this appendix we obtain some technical results about the spaces $C_\bullet(\ofn,V)$,  $R_\bullet(\ofn,V)$,  $W_\bullet(\ofn,V)$, $A_\bullet$ and $B_\bullet$ as introduced in Section \ref{Kostant}. Here $V$ is a finite dimensional $\mathfrak{osp}(1|2n)$-module, although the same results would hold for an arbitrary finite dimensional $\fh+\ofn$-module.

\begin{lemma}
\label{spacesAB}
The spaces $\{A_\bullet^{(j)}\}$ and $\{B_\bullet^{(k)}\}$ of Definition \ref{DefAB} are linearly independent. For $A_\bullet=\bigoplus_{j=1}^nA_\bullet^{(j)}$ and $B_\bullet=\bigoplus_{j=1}^nB_\bullet^{(j)}$ it holds that
\begin{eqnarray*}
W_\bullet(\ofn,V)=A_\bullet\oplus B_\bullet &\mbox{and}& A_k\cong B_{k-1}
\end{eqnarray*}
as $\fh$-modules  for $k\in\mathbb{N}$.
\end{lemma}
\begin{proof}
The monomials in the spaces $W_\bullet(\ofn,V)$, $A_\bullet$ and $B_\bullet$ form bases of these spaces. Therefore the proof can be written in terms of these monomials.

For every monomial in the span of the spaces $\{A_\bullet^{(j)}\}$ there is a certain $k$, such that it contains $Y^{\wedge 2}_{\delta_k}$ but no $Y_{2\delta_i}$ for $i\le k$, which separates this space from the span of the spaces $\{B_\bullet^{(j)}\}$. If for $a_j\in A_\bullet^{(j)}$, the element $\sum_{j=1}^n a_j$ is zero we can prove that every $a_j$ must be zero. If $k$ is the lowest number such that $a_k$ is not zero, then $a_k$ contains $Y^{\wedge 2}_{\delta_k}$ while none of the other terms contain this, therefore $a_k=0$.

Every monomial in $W_\bullet(\ofn,V)$ contains some term $Y_{2\delta_i}$ or some term $Y_{\delta_j}^{\wedge 2}$. If the lowest such $i$ is strictly lower than the lowest such $j$, this monomial is inside $A_\bullet$, if the lowest such $i$ is higher or equal to the lowest such $j$ the monomial is inside $B_\bullet$. This proves $W_\bullet(\ofn,V)=A_\bullet\oplus B_\bullet$.

Finally the morphism $A_k^{(j)}\to B_{k-1}^{(j)}$ defined by mapping $Y_{\delta_j}^{\wedge 2}\wedge f\to Y_{2\delta_j}\wedge f$ is clearly well-defined and bijective for every $j$.
\end{proof}

\begin{definition}
We introduce two subsets of the even positive roots $\Delta_{\oa}^+$ of $\mathfrak{osp}(1|2n)$
\begin{eqnarray*}
M=\{\delta_i-\delta_j|\forall i<j\}&\mbox{and}& P=\{\delta_i+\delta_j|\forall i<j\}.
\end{eqnarray*}
The grading $D$ on a monomial in $C_\bullet(\ofn,V)$ is defined as 
\begin{eqnarray*}
D(Y_{\alpha_1}\wedge\cdots\wedge Y_{\alpha_d}\otimes v)= \sharp \{\alpha_k\in M, k=1,\dots,d\}- \sharp \{\alpha_k\in P, k=1,\dots,d\}+D(v),
\end{eqnarray*}
where $D(v)=\sum_{i=1}^n\mu_i$ for $v$ a weight vector of weight $\sum_{i=1}^n\mu_i\delta_i$.
\end{definition}

Since the root vectors corresponding to the roots in $M$ and $P$ are even and $V$ is finite dimensional, the grading is finite and we define $\left(C_\bullet(\ofn,V)\right)[i]$ as the span of all the monomials $f$ that satisfy $D(f)=i$. Also for subspaces $L_\bullet\subset C_\bullet(\ofn,V)$ we set
\begin{eqnarray}
\label{Li}
\left(L_\bullet\right)[i]= \left(C_\bullet(\ofn,V)\right)[i] \cap L_\bullet
\end{eqnarray}

The following lemma follows immediately from the definition of the boundary operator.
\begin{lemma}
\label{neverrd}
The boundary operator $\delta^\ast:C_\bullet(\ofn,V)\to C_\bullet(\ofn,V)$ acting on a monomial $f$ with $D(f)=p$ yields $\delta^\ast f=\sum_j f_j$ for monomials $f_j$ that satisfy $D(f_j)\le p$.
\end{lemma}

The following calculation will be crucial for computing the cohomology.
\begin{lemma}
\label{calculation}
For $Y_{\delta_j}^{\wedge k}\wedge f\in C_\bullet(\ofn,V)$ the boundary operator acts as
\begin{eqnarray*}
\delta^\ast(Y_{\delta_j}^{\wedge k}\wedge f)&=&-\frac{1}{2}k(k-1)Y_{2\delta_j}\wedge Y_{\delta_j}^{\wedge k-2}\wedge f\\
&+&k(-1)^k Y_{\delta_j}^{\wedge k-1}\wedge Y_{\delta_j}\cdot f+(-1)^kY_{\delta_j}^{\wedge k}\wedge \delta^\ast f.
\end{eqnarray*}
\end{lemma}
\begin{proof}
From the immediate calculation
\begin{eqnarray*}
\delta^\ast(Y_{\delta_j}^{\wedge k}\wedge f)&=&-(k-1)Y_{2\delta_j}\wedge Y_{\delta_j}^{\wedge k-2}\wedge f\\
&+&(-1)^k Y_{\delta_j}^{\wedge k-1}\wedge Y_{\delta_j}\cdot f-Y_{\delta_j}\wedge\delta^\ast (Y_{\delta_j}^{\wedge k-1}\wedge f)
\end{eqnarray*}
the statement can be proven by induction on $k$.
\end{proof}

The previous results can now be brought together to come to the main conclusion of this appendix. The following result states that the coboundary operator maps the subspaces $A_\bullet$ bijectively to spaces isomorphic with $B_\bullet$.
\begin{theorem}
\label{thmisom}
The morphism
\begin{eqnarray*}
\phi: A_\bullet\to C_\bullet(\ofn,V)/(A_\bullet\oplus R_\bullet(\ofn,V))\cong B_\bullet
\end{eqnarray*}
given by the composition of the boundary operator $\delta^\ast :A_\bullet\to C_\bullet(\ofn,V)$ with the canonical projection onto $C_\bullet(\ofn,V)/(A_\bullet\oplus R_\bullet(\ofn,V))$ is an isomorphism.
\end{theorem}
\begin{proof}
First we prove that the morphism $\phi^{(l)}$ given by $\phi$ acting on the restriction to $(A_\bullet)[l]$ (as defined in equation \eqref{Li}) composed with the restriction 
\[C_\bullet(\ofn,V)/(A_\bullet\oplus R_\bullet(\ofn,V))\to\left(C_\bullet(\ofn,V)/(A_\bullet\oplus R_\bullet(\ofn,V))\right)[l]\] is an isomorphism. We take a general element of $(A_\bullet)[l]$ and expand it according to the decomposition $A_\bullet=\bigoplus_{j=1}^n A_\bullet^{(j)}$ in Definition \ref{DefAB}:
\begin{eqnarray*}
h&=&\sum_{j=1}^n\sum_{k=2}^{N_j} Y_{\delta_j}^{\wedge k}\wedge h_{k}^{(j)}
\end{eqnarray*}
where $h_{k}^{(j)}$ does not contain $Y_{\delta_j}$, $Y_{2\delta_j}$ or $Y_{\delta_i}^{\wedge 2}$ and $Y_{2\delta_i}$ for $i< j$ and $D(h_{k}^{(j)})=l$. According to Lemma \ref{calculation} the action of $\delta^\ast$ combined with projection onto $C_\bullet(\ofn,V)[l]$ is given by
\begin{eqnarray*}
\left(\delta^\ast h\right)[l]&=&-\frac{1}{2}\sum_{j=1}^n\sum_{k=2}^{N_j} k(k-1) Y_{2\delta_j }\wedge Y_{\delta_j}^{\wedge k-2}\wedge h_{k}^{(j)}\\
&+&\sum_{j=1}^n\sum_{k=2}^{N_j}(-1)^k Y_{\delta_j}^{\wedge k}\wedge\left(  \delta^\ast h_{k}^{(j)}\right)[l]
\end{eqnarray*}
since degree of the monomials in the terms $Y_{\delta_j}^{\wedge k-1}\wedge Y_{\delta_j}\cdot h_{k}^{(j)}$ is strictly lower than $l$. Assume $p$ is the smallest number for which $h_{p}^{(n)}$ is different from zero and assume that $h\in\ker\phi^{(l)}$. The term $Y_{2\delta_n }\wedge Y_{\delta_n}^{\wedge p-2}\wedge h_{p}^{(n)}$ is not inside $A_\bullet\oplus R_\bullet(\ofn,V)$ and $h_{p}^{(n)}$ does not contain $Y_{\delta_n}$ or any $Y_{2\delta_i}$ or $Y_{\delta_i}^{\wedge 2}$. Therefore there is no other term appearing in $\left(\delta^\ast h\right)[l]$ to compensate this one and we obtain $h_{k}^{(n)}\equiv0$ for every $k$. Then from similar arguments we obtain by induction that $h_{k}^{(j)}\equiv0$ must hold for every $j$ and $k$, so $\phi^{(l)}$ is injective. The isomorphism $A_k\cong B_{k-1}$ from Lemma \ref{spacesAB}, which can clearly be refined to $\left(A_k\right)[l]\cong\left( B_{k-1}\right)[l]$, then shows that injectivity implies surjectivity. 

Lemma \ref{neverrd} implies that $\delta^\ast$ never raises degree, a property that is immediately inherited by $\phi$. The combination of this with the fact that the grading is finite, leads to the conclusion that $\phi$ is bijective since the $\{\phi^{(l)}\}$ are.
\end{proof}

\section*{Appendix 2: Category $\cO$ for basic classical Lie superalgebras}
The BGG category $\cO$ for a basic classical Lie superalgebra $\fg$ is the full subcategory of the category of $\fg$-modules of modules $M$ that satisfy the conditions:
\begin{itemize}
\item $M$ is a finitely generated $\U(\fg)$-module.
\item $M$ is $\fh$-semisimple.
\item $M$ is locally $\U(\fn)$-finite.
\end{itemize}
In this appendix we mention some properties of this category which are needed in Section \ref{secBGG} and Section \ref{secpd}. For more details on category $\cO$ for Lie (super)algebras, see \cite{MR0578996, MR2100468, MR2428237, MR2366357, preprint}. We use notations similar to the rest of the paper, but now for arbitrary basic classical Lie superalgebras.

The following results are due to Mazorchuk, see Proposition 1 and Theorem 2 in \cite{preprint}, or Brundan, see Theorem 4.4 in \cite{MR2100468}.
\begin{lemma}
\label{projO}
In the category $\cO$ for basic classical Lie superalgebras each irreducible representation $L(\mu)$ has a projective cover and each projective module in $\cO$ has a standard filtration.
\end{lemma}
The projective cover of $L(\lambda)$ is denoted by $P(\lambda)$ and is also the projective cover of $M(\lambda)$.

\begin{lemma}[BGG reciprocity]
\label{BGGreproc} For a basic classical Lie superalgebra $\fg$ the following relation holds between the standard filtration of the projective module $P(\lambda)$ and the Jordan-H\"older series of the Verma module $M(\mu)$:
\[(P\left(\lambda):M(\mu)\right)=[M(\mu):L(\lambda)].\]
\end{lemma}
\begin{proof}
This is a special case of Corollary 4.5 in \cite{MR2100468}, but can also easily be proved directly. Firstly, we have $[M(\mu):L(\lambda)]=[M(\mu)^\vee:L(\lambda)].$ For any module $M\in\cO$ it holds that $[M:L(\lambda)]=\dim\Hom_{\cO}\left(P(\lambda),M\right)$, since this is true for $M$ irreducible and $\Hom_{\cO}\left(P(\lambda),-\right)$ is an exact functor and thus preserves short exact sequences. Together this yields
\[[M(\mu):L(\lambda)]=\dim\Hom_{\cO}\left(P(\lambda),M(\mu)^\vee\right).\]

The statement then follows from $\dim\Hom_{\cO}\left(P(\lambda),M(\mu)^\vee \right)=(P(\lambda):M(\mu))$, which can be proved similarly as Theorem 3.7 in \cite{MR2428237}.
\end{proof}

If an integral dominant weight is the highest one inside the class of weights corresponding to a central character (which is always true for typical highest weights) we obtain the classical result that the corresponding Verma module is projective.
\begin{lemma}
\label{Vermaproj}
Suppose $\Lambda\in\cP^+$ is the highest weight inside the set $\{\mu\in\fh^\ast|\chi_\mu=\chi_\Lambda\}$, then $M(\Lambda)$ is a projective module in $\cO$.
\end{lemma}
\begin{proof}
The proof does not change from the proof of Proposition 3.8 in \cite{MR2428237} because of the extra condition on $\chi_\Lambda$.
\end{proof}

\input{referenc}

\end{document}

%% file: referenc.tex
%
%
%
\biblstarthook{
}